\documentclass[11pt]{amsart}
\usepackage{amsmath,cite}
\usepackage[ps2pdf,colorlinks=true,urlcolor=blue,
citecolor=red,linkcolor=blue,linktocpage,pdfpagelabels,bookmarksnumbered,bookmarksopen]{hyperref}
\usepackage[english]{babel}

\usepackage[left=2.6cm,right=2.6cm,top=2.74cm,bottom=2.74cm]{geometry}

\numberwithin{equation}{section}
\newtheorem{theorem}{Theorem}[section]
\newtheorem{proposition}[theorem]{Proposition}

\newtheorem{remark}[theorem]{Remark}
\newtheorem{corollary}[theorem]{Corollary}
\newtheorem{definition}[theorem]{Definition}
\theoremstyle{definition}

\newcommand{\R}{{\mathbb R}}

\newcommand{\dvg}{{\rm div}}
\newcommand{\elle}[1]{L^{#1}(\Omega)}
\newcommand{\hsob}{H^1_0(\Omega)}
\newcommand{\js}[1]{j_s(x,{#1},\nabla {#1})}
\newcommand{\jxi}[1]{j_\xi(x,{#1},\nabla {#1})}
\newcommand{\into}{{\int_{\Omega}}}
\newcommand{\eps}{\varepsilon}
\newcommand{\dhsob}{H^{-1}(\Omega)}

\title[Diffeomorphism-invariant properties for quasi-linear operators]{Diffeomorphism-invariant 
properties \\ for quasi-linear elliptic operators}

\author[V.~Solferino]{Viviana Solferino}
\author[M.~Squassina]{Marco Squassina}

\address{Dipartimento di Matematica
\newline\indent
Universit\`a della Calabria
\newline\indent
Ponte Pietro Bucci 30B, I-87036 Arcavacata di Rende, Cosenza,
Italy} \email{solferino@mat.unical.it}

\address{Dipartimento di Informatica
\newline\indent
Universit\`a degli Studi di Verona
\newline\indent
C\'a Vignal 2, Strada Le Grazie 15, I-37134 Verona, Italy}
\email{marco.squassina@univr.it}

\thanks{Work supported by Miur project: {\em ``Variational and Topological
Methods in the Study of Nonlinear Phenomena''}}


\begin{document}

\begin{abstract}
For quasi-linear elliptic equations we detect relevant properties
which remain invariant under the action of a suitable class of diffeomorphisms.
This yields a connection between existence theories for equations with 
degenerate and non-degenerate coerciveness.
\end{abstract}

\maketitle
\noindent
{\em \small The second author wishes to dedicate the manuscript \newline to the memory of his mother Maria Grazia.}
\medskip

\section{Introduction}
Let $\Omega$ be a smooth bounded domain in $\R^N$. In the study of the nonlinear equation
\begin{equation}
\label{equation-intro}
-\dvg( j_\xi(x,u,\nabla u)) +j_s(x,u,\nabla u)=g(x,u)\quad\,\,\, \text{in $\Omega$},
\end{equation}
an important r\v ole is played by the coerciveness feature of $j$, namely the fact 
that there exists a positive constant $\sigma>0$ such that 
\begin{equation}
	\label{coerciv}
j(x,s,\xi)\geq \sigma |\xi|^2,\quad\,\,\text{for a.e. $x\in\Omega$ and all $(s,\xi)\in\R\times\R^N$}.
\end{equation}
Under condition \eqref{coerciv} and other suitable assumptions, including the 
boundedness of the map $s\mapsto j(x,s,\xi)$, 
equation \eqref{equation-intro} has been deeply investigated in the last twenty
years by means of variational methods and tools of non-smooth critical point theory, essentially via
two different approaches (see e.g.~\cite{Ab} and \cite{candeg} and references therein).
More recently, it was also covered the case where the map $s\mapsto j(x,s,\xi)$ is unbounded
(see e.g.~\cite{Ab1} and \cite{pelsqu}, again via different strategies). The situation is by far
more delicate under the assumption of degenerate coerciveness, namely for some function $\sigma:\R\to\R^+$
with $\sigma(s)\to 0$ as $s\to\infty$,
\begin{equation}
	\label{noncoerciv}
j(x,s,\xi)\geq \sigma(s)|\xi|^2,\quad\,\,\text{for a.e. $x\in\Omega$ and all $(s,\xi)\in\R\times\R^N$}.
\end{equation}
To the authors' knowledge, in this setting, for $j$ of the form $(b(x)+|s|)^{-2\beta}|\xi|^2/2$,
the first contribution to minimization problems is \cite{BO}, while for 
existence of mountain pass type solutions we refer to \cite{AbO}, the main point
being the fact that cluster points of arbitrary Palais-Smale sequences are bounded.
See \cite{albofeortr} for more general existence statements and
\cite{BDO,BoBr} for regularity results.

Relying upon a solid background for the treatment of \eqref{equation-intro} in the coercive case, 
the main goal of this paper is that of building
a bridge between the theory for non-degenerate coerciveness problems 
and that for problems with degenerate coerciveness. Roughly speaking, we see a solution to a degenerate problem as related
to a solution of a corresponding non-degenerate problem, preserving at the same time the main structural assumptions
typically assumed for these classes of equations. To this aim, we introduce a suitable class of diffeomorphisms
$\varphi\in C^2(\R)$ and consider the functions $j^\sharp:\Omega\times\R\times\R^N\to\R$ and $g^\sharp:\Omega\times\R\to\R$,
defined as
$$
j^\sharp(x,s,\xi)=j(x,\varphi(s),\varphi'(s)\xi), \qquad
g^\sharp(x,s)=g(x,\varphi(s))\varphi'(s),
$$
for a.e. $x\in\Omega$ and all $(s,\xi)\in\R\times\R^N$. Then, if \eqref{noncoerciv} holds,
we can find $\sigma^\sharp>0$ such that
\begin{equation*}
j^\sharp(x,s,\xi)\geq \sigma^\sharp |\xi|^2,
\end{equation*}
for a.e.~$x\in\Omega$ and all $(s,\xi)\in\R\times\R^N$, thus recovering the non-degenerate coerciveness
from the original degenerate framework. 
We shall write the corresponding Euler's equation as 
\begin{equation}
\label{equation-intro-sharp}
-\dvg( j_\xi^\sharp(x,v,\nabla v)) +j_s^\sharp(x,v,\nabla v)=g^\sharp(x,v) \quad\,\,\, \text{in $\Omega$}.
\end{equation}
A first natural issue is the correspondence between the solutions of \eqref{equation-intro}
and the solutions of \eqref{equation-intro-sharp} through the diffeomorphism $\varphi$. Roughly speaking, the natural connection
is that $u=\varphi(v)$ is a solution of \eqref{equation-intro} when
$v$ is a solution to \eqref{equation-intro-sharp}, in some sense.
On the other hand, in general, $\varphi(v)\not\in  H^1_0(\Omega)$ although $v\in H^1_0(\Omega)$. Hence, the notion of solution
for functions in the Sobolev space $H^1_0(\Omega)$ cannot remain invariant under the action of $\varphi$, unless 
$v\in L^\infty(\Omega)$. In fact, we provide a new 
definition of generalized solution which is partly based
upon the notion of renormalized solution introduced in \cite{DMOP} 
in the study of elliptic equations with general measure data and partly on the variational formulation adopted in \cite{pelsqu}.
The new notion turns out to be invariant under diffeomorphisms (Proposition~\ref{soluzdisrt}) as well as conveniently related
to the machinery developed in \cite{pelsqu}.
Moreover, we detect two relevant invariant conditions.
The first (Proposition~\ref{invar1}) is a modification of the standard (non-invariant) 
sign condition 
\begin{equation}
	\label{classicsign}
j_s(x,s,\xi)s\geq 0,\quad \text{for all $|s|\geq R$ and some $R\geq 0$},
\end{equation}
namely there exist $\eps\in (0,1)$ and $R\geq 0$ such that
\begin{equation}
	\label{generalsign}
(1-\eps) j_\xi(x,s,\xi)\cdot\xi+j_s(x,s,\xi)s\geq 0,
\end{equation}
for a.e.~$x\in\Omega$ and all $(s,\xi)\in\R\times\R^N$ such that $|s|\geq R$. Condition~\eqref{classicsign}
is well known \cite{Ab,Ab1,AbO,candeg,pelsqu} and plays an important r\v ole in the study of both 
existence and summability issues for \eqref{equation-intro}.
The second one (Proposition~\ref{invar2}) is the generalized Ambrosetti-Rabinowitz \cite{AR} condition:
there exist $\delta>0$, $\nu>2$ and $R\geq 0$ such that
\begin{equation}
	\label{genARcondition}
\nu j(x,s,\xi)-(1+\delta)j_\xi(x,s,\xi)\cdot\xi-j_s(x,s,\xi)s-\nu G(x,s)+g(x,s)s\geq 0,
\end{equation}
for a.e.~$x\in\Omega$ and all $(s,\xi)\in\R\times\R^N$ with $|s|\geq R$.
Typically, this condition guarantees that an arbitrary Palais-Smale sequence is bounded 
\cite{Ab,Ab1,candeg,pelsqu}. The invariant properties for growth conditions are stated in
Proposition~\ref{rm1}, \ref{menoinf} and \ref{newgrow}. In the situations where
\begin{equation*}
j_s^\sharp(x,s,\xi)s\geq 0,\quad \text{for all $|s|\geq R^\sharp$ and some $R^\sharp\geq 0$},
\end{equation*}
the results of our paper allow to obtain existence and multiplicity
of solutions for problems with degenerate coercivity by a {\em direct} 
application of the results of \cite{pelsqu} (see Theorem~\ref{existthmm}). This is new 
compared with the results of \cite{AbO}, since the technique adopted therein does not allow
to obtain multiplicity results. In addition, contrary to \cite{AbO}, under certain assumptions on the nonlinearity
$g$, the solutions need not to be bounded.
The further development of the ideas in this paper, is related to strengthening some of the results
of \cite{pelsqu}, in order to allow the weaker sign condition \eqref{generalsign} 
to replace the standard sign condition \eqref{classicsign}. 
Then existence and multiplicity theorems for coercive equations with unbounded coefficients
automatically recover existence and multiplicity theorems for equations 
with degenerate coercivity. This will be the subject of a further investigation.
\vskip4pt
\noindent
The plan of the paper is as follows.
\newline In Section~\ref{generalizedsol} we introduce 
a new notion of generalized solution for \eqref{equation-intro} and prove that 
it is invariant under the action of $\varphi$. In Section~\ref{growthsect} we show how
$\varphi$ affects some useful growth conditions.
In Section~\ref{signsect} we study the invariance of the sign condition \eqref{generalsign}
and get some related summability results. In Section~\ref{ambrrabsect}, we consider the invariance of an Ambrosetti-Rabinowitz
(AR, in brief) type inequality \eqref{genARcondition}. Finally, in Section~\ref{existence-sect}, we shall get a new existence results for multiple,
possibly unbounded, generalized solutions of \eqref{equation-intro}.

\section{Invariant properties}

Now let $\Omega$ be a smooth bounded domain in $\R^N$. We consider
$j:\Omega\times\R\times\R^N\to\R$ with
$j(\cdot,s,\xi)$ measurable in $\Omega$ for all $s\in\R$ and
$\xi\in\R^N$ and $j(x,\cdot,\cdot)$ of class $C^1$ for
a.e.~$x\in\Omega$. Moreover, we assume that the map
$\xi\mapsto j(x,s,\xi)$ is strictly convex and
there exist $\alpha,\gamma,\mu:\R^+\to\R^+$ continuous 
with $\alpha(s)\geq 1$ for all $s\in\R^+$ and such that
\begin{gather}
    \label{originalgrowths1}
\frac{1}{\alpha(|s|)} |\xi|^2  \leq j(x,s,\xi)\leq \alpha(|s|)|\xi|^2,   \\
\label{originalgrowths2}
|j_s(x,s,\xi)|\leq \gamma(|s|)|\xi|^2, \quad\quad
|j_\xi(x,s,\xi)|\leq\mu(|s|)|\xi|,
\end{gather}
for a.e.~$x\in\Omega$ and all $(s,\xi)\in\R\times\R^N$. Actually, the second inequality of \eqref{originalgrowths2}
can be deduced by the strict convexity of $\xi\mapsto j(x,s,\xi)$ and the right inequality of \eqref{originalgrowths1}.
Furthermore, again by the strict convexity of $\xi\mapsto j(x,s,\xi)$ and the left inequality of \eqref{originalgrowths1} it holds
\begin{equation}\label{jxicoerc}
j_\xi(x,s,\xi)\cdot \xi\geq \frac{1}{\alpha(|s|)} |\xi|^2,
\end{equation}
see \cite[Remarks 4.1 and 4.3]{pelsqu}.
Without loss of generality, one may assume that 
$\alpha,\gamma,\mu:\R^+\to\R^+$ appearing in the
growth conditions of $j,j_s,j_\xi$ are monotonically
increasing. Indeed, we can always replace them 
by the increasing functions $\alpha_0,\gamma_0,\mu_0:\R^+\to\R^+$ defined by
\begin{equation*}
\alpha_0(r)=\sup\limits_{s\in [-r,r]} \alpha(|s|), \quad\,\,
\gamma_0(r)=\sup\limits_{s\in [-r,r]} \gamma(|s|),\quad\,\,
\mu_0(r)=\sup\limits_{s\in [-r,r]} \mu(|s|).
\end{equation*}
We shall also assume that $g:\Omega\times\R\to\R$ is a Carath\'eodory function such that
\begin{equation}
\label{sumbgcond}
	\sup_{|t|\leq s}|g(\cdot,t)|\in L^1(\Omega),\quad\,\,\text{for every $s\in\R^+$,}
\end{equation}
and we set $G(x,s)=\int_0^s g(x,t)dt$, for every $s\in\R$.

\begin{definition}
    \label{diffeoclass}
For an odd diffeomorphism $\varphi:\R\to\R$ of class $C^2$ such that $\varphi(0)=0$, 
we consider the following properties 
  \begin{align}
    \label{one}
    \varphi'(s)\geq \sigma\sqrt{\alpha(|\varphi(s)|)},\qquad\text{for all $s\in\R$ and some $\sigma>0$}. \\
        \label{two}
 \lim_{s\to + \infty}
\frac{s\varphi'(s)}{\varphi(s)}= 1+\lim_{s\to + \infty}
\frac{s\varphi''(s)}{\varphi'(s)}=\frac{1}{1-\beta},\qquad\text{for some $\beta\in [0,1)$.}
\end{align}
\end{definition}

\noindent
A simple model satisfying the requirements of Definition~\ref{diffeoclass} is the function
\begin{equation}
	\label{fimodel}
\varphi(s)=s{(1+s^2)}^{\frac{\beta}{2(1-\beta)}},\qquad\text{for
all $s\in\R$}, \qquad 0\leq \beta<1,
\end{equation}
in the case when $\alpha(t)=C(1+t)^{2\beta}$, for some $C>0$.

\begin{definition}
Consider the functions
$$
j:\Omega\times\R\times\R^N\to\R,\quad
g:\Omega\times\R\to\R, \quad
G:\Omega\times\R\to\R,
$$
and let
$\varphi\in C^2(\R)$ be a diffeomorphism according to Definition~\ref{diffeoclass}. We define
$$
j^\sharp:\Omega\times\R\times\R^N\to\R,\quad
g^\sharp:\Omega\times\R\to\R, \quad
G^\sharp:\Omega\times\R\to\R,
$$
by setting
$$
j^\sharp(x,s,\xi)=j(x,\varphi(s),\varphi'(s)\xi),
$$
for a.e. $x\in\Omega$ and all $(s,\xi)\in\R\times\R^N$ and
$$
g^\sharp(x,s)=g(x,\varphi(s))\varphi'(s), \qquad
G^\sharp(x,s)=\int_0^{s} g^\sharp(x,t)dt=G(x,\varphi(s)),
$$
for a.e. $x\in\Omega$ and all $s\in\R$.
\end{definition}

\noindent
Now we see that $\varphi$ turns a degenerate problem associated with $j$ into a non-degenerate one,
associated with $j^\sharp$ and that $j^\sharp,j_s^\sharp$ and $j_\xi^\sharp$ satisfy 
growths analogous to those of $j,j_s$ and $j_\xi$.

\begin{proposition}
    \label{rm1}
	Let $\varphi\in C^2(\R)$ be a diffeomorphism which satisfies 
	the properties of Definition~\ref{diffeoclass}.
Assume that $\alpha,\gamma,\mu:\R^+\to\R^+$ satisfy the growth conditions~\eqref{originalgrowths1}-\eqref{originalgrowths2}.
Then there exist continuous functions $\alpha^\sharp,\gamma^\sharp,\mu^\sharp:\R^+\to\R^+$ and $\sigma^\sharp>0$ such that
\begin{gather*}
\sigma^\sharp |\xi|^2 \leq j^\sharp(x,s,\xi)\leq \alpha^\sharp(|s|)|\xi|^2, \\
    \noalign{\vskip2pt}
    |j^\sharp_s(x,s,\xi)|\leq  \gamma^\sharp(|s|)|\xi|^2,  \quad\,\,\,
    |j^\sharp_\xi(x,s,\xi)|\leq  \mu^\sharp(|s|)|\xi|,
\end{gather*}
for a.e.~$x\in\Omega$ and all $(s,\xi)\in\R\times\R^N$.
\end{proposition}
\begin{proof}
In light of \eqref{originalgrowths1} and of \eqref{one} of Definition~\ref{diffeoclass}, for $\sigma^\sharp=\sigma^2$, we have
    \begin{equation*}
\sigma^\sharp|\xi|^2\leq \frac{\varphi'(s)^2}{\alpha(|\varphi(s)|)}|\xi|^2\leq j(x,\varphi(s),\varphi'(s)\xi) 
\leq \alpha(|\varphi(s)|)\varphi'(s)^{2}|\xi|^2
    \end{equation*}
for a.e.~$x\in\Omega$ and all $(s,\xi)\in\R\times\R^N$. Furthermore, 
by virtue of \eqref{originalgrowths2}, we have
$$
|j^\sharp_\xi(x,s,\xi)|\leq (\varphi'(s))^2\mu(|\varphi(s)|)|\xi|,
$$
for a.e.~$x\in\Omega$ and all $(s,\xi)\in\R\times\R^N$, as well as
\begin{equation*}
|j^\sharp_s(x,s,\xi)| \leq 	[|\varphi''(s)| \mu(|\varphi(s)|)\varphi'(s)
	+(\varphi'(s))^3\gamma(|\varphi(s)|)]|\xi|^2,
\end{equation*}
for a.e.~$x\in\Omega$ and all $(s,\xi)\in\R\times\R^N$. The
assertions follow with $\alpha^\sharp,\gamma^\sharp,\mu^\sharp:\R\to\R^+$,
\begin{align*}
\alpha^\sharp(s) & =\alpha(|\varphi(s)|)\varphi'(s)^2,   \\
\gamma^\sharp(s)  & =|\varphi''(s)|\mu(|\varphi(s)|)\varphi'(s)+(\varphi'(s))^3\gamma(|\varphi(s)|),   \\
\mu^\sharp (s) & =(\varphi'(s))^2\mu(|\varphi(s)|),
\end{align*}
for all $s\in\R$. Of course, without loss of generality, one can then substitute $\alpha^\sharp,\gamma^\sharp,\mu^\sharp$ with
even functions satisfying the same growth controls. 
\end{proof}

\subsection{Generalized solutions}
\label{generalizedsol}

For any $k>0$, consider the truncation $T_k:\R\to\R$,
$$
T_k(s)
=
\begin{cases}
    s & \text{for $|s|\leq k$}, \\
    k\,{\rm sign}(s)   & \text{for $|s|\geq k$}.
\end{cases}
$$
Moreover, as in \cite{pelsqu}, for a measurable function $u:\Omega\to\R$,
let us consider the space
\begin{equation}
    \label{sottopiccolo}
V_u=\big\{v\in H^1_0(\Omega)\cap L^\infty(\Omega): u\in L^\infty(\{v\neq 0\})\big\}.
\end{equation}
This functional space was originally introduced by Degiovanni and Zani 
for functions $u$ of $H^1_0(\Omega)$, in which case $V_u$ turns out to be a dense subspace 
of $H^1_0(\Omega)$ (cf.\ \cite{DZ}). Observe that, in view of conditions \eqref{originalgrowths2} 
and \eqref{sumbgcond}, it follows
$$
\jxi{u}\cdot \nabla v \in L^1(\Omega),\quad\,\,
\js{u}v\in L^1(\Omega),\quad\,\, g(x,u)v\in L^1(\Omega),
$$
for every $v\in V_u$ and any measurable $u:\Omega\to\R$ with 
$T_k(u) \in H^1_0(\Omega)$ for every $k>0$. For such functions, according to \cite{DMOP},
the meaning of $\nabla u$ will be made clear in the proof of Proposition~\ref{soluzdisrt}.
\vskip2pt
\noindent
In the spirit of \cite{DMOP}, where the notion of renormalized solution
is introduced, and \cite{pelsqu}, where the notion of generalized solution
is given, based upon $V_u$, we now introduce the following 

\begin{definition}
    \label{defsol-bis}
We say that $u$ is a generalized solution to
\begin{equation}\label{problema-bis}
\begin{cases}
-\,\dvg( j_\xi(x,u,\nabla u)) +j_s(x,u,\nabla u)=g(x,u), & \text{in $\Omega$}, \\
\quad u=0, & \text{on $\partial\Omega$},
\end{cases}
\end{equation}
if $u$ is a measurable function finite almost everywhere, such that
\begin{equation}
	\label{trunk-sol}
T_k(u) \in H^1_0(\Omega),\quad\text{for all $k>0$},
\end{equation}
and, furthermore,
\begin{equation}
	\label{summab-solger}
\jxi{u}\cdot\nabla u\in\elle1,\qquad \js{u}u\in \elle1,
\end{equation}
and
\begin{equation}
	\label{condizioneVu}
\into \jxi{u}\cdot \nabla w +\into
\js{u}w=\into g(x,u)w, \quad\forall w\in V_u.
\end{equation}
\end{definition}

\begin{remark}\rm
We point out that, in \cite[Definition 1.1]{pelsqu}, a different notion of generalized 
solution of problem \eqref{problema-bis} is introduced
when $u$ belongs to the Sobolev space $H^1_0(\Omega)$. On the other hand, actually, 
by \cite[Theorem 4.8]{pelsqu} the two notions agree, whenever $u\in H^1_0(\Omega)$.
Also, the variational formulation \eqref{condizioneVu} with test functions in $V_u$ is 
conveniently related to the weak slope \cite{DM,CDM} of the functional associated with \eqref{problema-bis}, see
\cite[Proposition 4.5]{pelsqu} (see also Proposition \ref{sommab}).
\end{remark}

\noindent
The following proposition establishes a link between the generalized solutions
of the problem under the change of variable procedure.

\begin{proposition}\label{soluzdisrt}
Let $\varphi\in C^2(\R)$ be a diffeomorphism which satisfies 
the properties of Definition~\ref{diffeoclass}.
Assume that $v$ is a generalized solution to
\begin{equation}
    \label{probmoddd}
\begin{cases}
-\dvg(j^\sharp_\xi(x,v,\nabla v)) +j^\sharp_s(x,v,\nabla v)=g^\sharp(x,v) & \text{in $\Omega$}, \\
\quad v=0,  & \text{on $\partial\Omega$}.
\end{cases}
\end{equation}
Then $u=\varphi(v)$ is a generalized solution to
\begin{equation}
    \label{ooorigg}
\begin{cases}
-\dvg( j_\xi(x,u,\nabla u)) +j_s(x,u,\nabla u)=g(x,u), & \text{in $\Omega$}, \\
\quad u=0, & \text{on $\partial\Omega$}.
\end{cases}
\end{equation}
If in addition $v\in H^1_0\cap L^\infty(\Omega)$, 
then $u\in H^1_0\cap L^\infty(\Omega)$ is a distributional solution to~\eqref{ooorigg}.
\end{proposition}
\begin{proof}
	As proved in \cite{DMOP}, for a measurable function $u$ on $\Omega$, finite almost everywhere,
with $T_k(u)\in H^1_0(\Omega)$ for any $k>0$, there exists 
a unique $\omega:\Omega\to\R^N$, measurable and such that
\begin{equation}
	\label{gradientcharacteriz}
\nabla T_k(u)=\omega\chi_{\{|u|\leq k\}},\quad \text{almost everywhere in $\Omega$ and for all $k>0$.}	
\end{equation}
Then, the gradient $\nabla u$ of $u$ is naturally defined by setting $\nabla u=\omega$. 
Assume that $\varphi:\R\to\R$ is a diffeomorphism with $\varphi(0)=0$ and that
for a measurable function $v$ on $\Omega$ it holds $T_k(v)\in H^1_0(\Omega)$ for every $k>0$. Then, setting $u=\varphi(v)$,
it follows $T_k(u)\in H^1_0(\Omega)$ for every $k>0$. In fact, given $k>0$, there exists $h>0$ such that
$T_k(u) = (T_k\circ\varphi)\circ T_{h}(v)$. Since $T_k\circ\varphi:\R\to\R$ is a globally
Lipschitz continuous function which is zero at zero, it follows that
$T_k(u)\in H^1_0(\Omega)$ for all $k>0$. 
Moreover, if $\nabla u$ and $\nabla v$ denote
the gradients of $u$ and $v$ respectively, in the sense pointed out above, we get the following chain rule 
\begin{equation}
	\label{chainrule}
\nabla u=\varphi'(v)\nabla v,\quad \text{almost everywhere in $\Omega$}.
\end{equation}
In fact, for all $k>0$, since $T_k(u),T_h(v)\in H^1_0(\Omega)$, from $T_k(u) = (T_k\circ\varphi)\circ T_{h}(v)$
we can write
$$
\nabla T_k(u)=(T_k\circ\varphi)'(T_{h}(v))\nabla T_h(v),
$$
for every $k>0$, namely, by \eqref{gradientcharacteriz},
\begin{equation}
\label{gradprimaff} 
\nabla u\chi_{\{|\varphi(v)|\leq k\}}=(T_k\circ\varphi)'(T_{h}(v))\nabla v\chi_{\{|v|\leq h\}},
\quad \text{almost everywhere in $\Omega$}.
\end{equation}
Let now $x\in\Omega$ be an arbitrary point with $|v(x)|\leq h$. In turn, by construction, $|\varphi(v(x))|\leq k$, 
and formula~\eqref{gradprimaff} yields directly
\begin{equation}
\label{gradprimaff-2} 
\nabla u=(T_k\circ\varphi)'(v)\nabla v,
\quad \text{almost everywhere in $\{|v|\leq h\}$}.
\end{equation}
Formula \eqref{chainrule} then follows by taking into account that $(T_k\circ\varphi)'(v(x))=\varphi'(v(x))$
almost everywhere in $\{|v|\leq h\}$ and by the arbitrariness of $h>0$.

Let now $v$ be a generalized solution to~\eqref{probmoddd}, so that $T_k(v)\in H^1_0(\Omega)$ for all $k>0$.
As pointed out above, it follows that $T_k(u)\in H^1_0(\Omega)$ too, for every $k>0$
and the chain rule $\nabla u=\varphi'(v)\nabla v$ holds, almost everywhere in $\Omega$. 
From the definition of generalized solution we learn that
\begin{equation}
    \label{startingsum}
 j^\sharp_{\xi}(x,v,\nabla v)\cdot\nabla v\in\elle1, \qquad
j^\sharp_s(x,v,\nabla v)v\in \elle1,
\end{equation}
as well as
\begin{equation}
    \label{solut1}
\into j^\sharp_\xi(x,v,\nabla v)\cdot \nabla w+\into
j^\sharp_s(x,v,\nabla v) w=\into g^\sharp(x,v)\,w, \quad \forall
w\in V_v.
\end{equation}
Notice that, for any $w\in V_v$, the integrands in \eqref{solut1} are in $L^1(\Omega)$, by
Proposition \ref{rm1}, the definition of $V_v$ and $\nabla v=\nabla T_k(v)
\in L^2(\{w\neq 0\})$ for any $k>\|v\|_{L^\infty(\{w\neq 0\})}$. In light of 
\eqref{chainrule} and \eqref{startingsum}, it follows that
$$
j_{\xi}(x,u,\nabla u)\cdot \nabla u =j^\sharp_{\xi}(x,v,\nabla
v)\cdot\nabla v\in\elle1.
$$
Moreover, a simple computation yields
$$
j^\sharp_s(x,v,\nabla
v)v=\Big[\frac{v\varphi'(v)}{\varphi(v)}\chi_{\{v\neq 0\}}\Big]j_s(x,u,\nabla u)u
+\Big[\frac{v\varphi''(v)}{\varphi'(v)}\Big]j_\xi(x,u,\nabla
u)\cdot\nabla u.
$$
Hence, in view of \eqref{two}, it follows that
$j_s(x,u,\nabla u)u\in\elle1$,
being $j_{\xi}(x,u,\nabla u)\cdot \nabla u\in\elle1$ and $
j^\sharp_s(x,v,\nabla v)v\in\elle1$. This yields the desired
summability conditions. For any $w \in V_v$, consider now $\hat
w=\varphi'(v)w$. We have $\hat w\in V_u$. In fact, since 
$v\in L^\infty(\{w\neq 0\})$, we obtain 
$\hat w\in H^1_0(\Omega)\cap L^\infty(\Omega)$ and 
$u=\varphi(v)\in L^\infty(\{w\neq 0\})=L^\infty(\{\hat w\neq 0\})$, 
since $\varphi'$ is positive by virtue of \eqref{one}. Of course, we have
$\hat w=\varphi'(T_k(v))w$, for all $k>\|v\|_{L^\infty(\{w\neq 0\})}$. Hence, recalling \eqref{gradientcharacteriz}, from
$$
\nabla (\varphi'(T_k(v))w)=w\varphi''(T_k(v))\nabla v\chi_{\{|v|\leq k\}}+\varphi'(T_k(v))\nabla w,\quad\text{for any $k>0$,}
$$
by choosing $k>\|v\|_{L^\infty(\{w\neq 0\})}$, we conclude that
$$
\nabla \hat w=w\varphi''(v)\nabla v+\varphi'(v)\nabla w,\quad \text{almost everywhere in $\Omega$.}
$$
Therefore, by easy computations, we get 
\begin{align}
    \label{primaeqqq}
j_\xi(x,u,\nabla u)\cdot\nabla \hat w &= j^\sharp_{\xi}(x,v,\nabla
v)\cdot\nabla w+
\frac{\varphi''(v)w}{\varphi'(v)} j_\xi(x,u,\nabla u)\cdot\nabla u, \\
\label{secondaeqqq} j_s(x,u,\nabla u)\hat w &=
j^\sharp_s(x,v,\nabla v) w
-\frac{\varphi''(v)w}{\varphi'(v)} j_\xi(x,u,\nabla
u)\cdot\nabla u,
\end{align}
yielding
\begin{equation*}
    %\label{summabill}
j_{\xi}(x,u,\nabla u)\cdot\nabla \hat w\in\elle1,\qquad
j_s(x,u,\nabla u)\hat w\in \elle1,
\end{equation*}
since $j^\sharp_{\xi}(x,v,\nabla v)\cdot\nabla w\in\elle1$,
$j^\sharp_s(x,v,\nabla v)w\in \elle1$ and
$$
\int_\Omega \big|\frac{\varphi''(v)w}{\varphi'(v)}\, j_\xi(x,u,\nabla u)\cdot\nabla u\big|
=\int_{\{w\neq 0\}} \big|\frac{\varphi''(v)w}{\varphi'(v)}\, j_\xi(x,u,\nabla u)\cdot\nabla u\big|
\leq C\int_\Omega \big|j_\xi(x,u,\nabla u)\cdot\nabla u\big|.
$$
By adding identities \eqref{primaeqqq}-\eqref{secondaeqqq} and
recalling the definition of $g^\sharp(x,v)$, we get
from~\eqref{solut1}
\begin{equation*}
%   \label{solut2}
\into j_\xi(x,u,\nabla u)\cdot \nabla \hat w+\into j_s(x,u,\nabla u)
\hat w=\into g(x,u)\hat w, \quad \text{$\hat w=\varphi'(v)w\in V_u$.}
\end{equation*}
Given any $z\in V_u$, we have $w=\frac{z}{\varphi'(v)}=
\frac{z}{\varphi'(T_k(v))}\in V_v$ for $k>\|v\|_{L^\infty(\{z\neq 0\})}$. In turn, 
\begin{equation*}
%   \label{solut2}
\into j_\xi(x,u,\nabla u)\cdot \nabla z+\into
j_s(x,u,\nabla u) z=\into g(x,u)z,\qquad\text{for every $z \in V_u$,}
\end{equation*}
yielding the assertion. Finally, if $v$ is a bounded generalized
solution to~\eqref{probmoddd}, $u\in H^1_0(\Omega)$ is bounded too and it follows that
$u=\varphi(v)$ is a distributional solution to~\eqref{ooorigg}.
\end{proof}

\begin{remark}\rm
The gradient $\nabla u=\omega$ does not agree, in general, with the one in the sense of distributions, since
it could be either $u\not\in L^1_{{\rm loc}}(\Omega)$ or $\omega\not\in L^1_{{\rm loc}}(\Omega,\R^N)$. 
If $\omega\in L^1_{{\rm loc}}(\Omega,\R^N)$, then $u\in W^{1,1}_{{\rm loc}}(\Omega)$ 
and $\omega$ agrees with the distributional gradient \cite[Remark 2.10]{DMOP}.
\end{remark}

\noindent
Under natural regularity assumptions, a generalized solution is, actually, distributional. 

\begin{proposition}
\label{soldistr}	
Assume that $u$ is a generalized solution to problem~\eqref{problema-bis} and that, in addition
\begin{equation}
	\label{summab-extra}
\jxi{u}\in L^1_{{\rm loc}}(\Omega;\R^N),\qquad \js{u}\in L^1_{{\rm loc}}(\Omega),\qquad g(x,u)\in L^1_{{\rm loc}}(\Omega).
\end{equation}
Then $u$ solves problem~\eqref{problema-bis} in the sense of distributions.
\end{proposition}
\begin{proof}
	Let $H:\R\to\R$ be a smooth cut-off function such that 
	$0\leq H\leq 1$, $H(s)=1$ for $|s|\leq 1$ and $H(s)=0$ for $|s|\geq 2$.
	Given $k>0$ and $\varphi\in C^\infty_c(\Omega)$, consider 
	in formula \eqref{condizioneVu} the admissible test functions $w=w_k=H(T_{2k+1}(u)/k)\varphi\in V_u$.
	Whence, for every $k>0$, it holds that
	\begin{align}
		\label{quasi-distr}
	&\into \jxi{u}\cdot H(T_{2k+1}(u)/k)\nabla \varphi+\into \jxi{u}\cdot H'(T_{2k+1}(u)/k)1/k \nabla T_{2k+1}(u)\varphi  \notag \\
	&+\into \js{u}H(T_{2k+1}(u)/k)\varphi=\into g(x,u)H(T_{2k+1}(u)/k)\varphi.
	\end{align}
	Taking into account that $\jxi{u}\cdot \nabla u \in L^1(\Omega)$ and 
	by \eqref{gradientcharacteriz}, for all $k>0$ we have
	$$
	|\jxi{u}\cdot H'(T_{2k+1}(u)/k)1/k\nabla T_{2k+1}(u)\varphi|\leq C|\jxi{u}\cdot \nabla u| \in L^1(\Omega),
	$$
	yielding, by the Dominated Convergence Theorem,
	$$
	\lim_k \into \jxi{u}\cdot H'(T_{2k+1}(u)/k)1/k \nabla T_{2k+1}(u)\varphi=0.
	$$
	On account of assumptions \eqref{summab-extra}, the assertion follows by letting $k\to\infty$ into \eqref{quasi-distr},
	again in light of the Dominated Convergence Theorem.
\end{proof}

\subsection{Further growth conditions}
\label{growthsect}

The next proposition is useful for the study of the mountain pass geometry of the functional
associated with problem \eqref{equation-intro}.

\begin{proposition}
    \label{menoinf}
	Let $\varphi\in C^2(\R)$ be a diffeomorphism satisfying
	the properties of Definition~\ref{diffeoclass} and such that
\begin{equation}
	\label{diffeoasymptoticbb}
0<\lim_{s\to+\infty}\frac{\varphi(s)}{s^{\frac{1}{1-\beta}}}<+\infty,
\end{equation}
and let $\alpha^\sharp: \R^+\to\R^+$ be the function  
introduced in Proposition~\ref{rm1}. Let $\nu>2(1-\beta)$, 
$k_1\in L^{\infty}(\Omega)$ with $k_1>0$, $k_2\in \elle1$, $k_3\in \elle{2N/(N+2)}$. Assume that
\begin{equation}
	\label{cdpdmm}
    \lim_{s\to\infty}\frac{\alpha(|s|)}{|s|^{\nu-2}}=0
\quad
\text{and}
\quad
G(x,s)\geq k_1(x)|s|^{\nu}-k_2(x)-k_3(x)|s|^{1-\beta},
\end{equation}
for a.e.~$x\in\Omega$ and all $s\in\R$. Then there exist $\nu^\sharp>2$ such that
$$
\lim_{s\to\infty}\frac{\alpha^\sharp(|s|)}{|s|^{\nu^\sharp-2}}=0
\quad \text{and} \quad G^\sharp (x,s)\geq k^\sharp_1(x)|s|^{
\nu^\sharp}-k_2^\sharp(x)- k_3^\sharp(x)|s|,
$$
for a.e.~$x\in\Omega$ and all $s\in\R$, for some 
$k_1^\sharp\in L^{\infty}(\Omega)$, $k_1^\sharp>0$,
$k_2^\sharp\in L^1(\Omega)$ and
$k^\sharp_3\in \elle{\frac{2N}{N+2}}$.
\end{proposition}

\begin{proof}
By assumption \eqref{diffeoasymptoticbb} and \eqref{two}, for $\nu^\sharp=\frac{\nu}{1-\beta}$, we have
    \begin{equation*}
        \lim_{s\to+\infty}\frac{\alpha^\sharp(s)}{s^{\nu^\sharp-2}}
         =
        \lim_{s\to\infty}\frac{\alpha(\varphi(s))}{\varphi(s)^{\nu-2}}\cdot
        \lim_{s\to\infty}\frac{\varphi(s)^{\nu-2}\varphi'(s)^2}{s^{\nu^\sharp-2}}
        =0.
    \end{equation*}
Finally, if $G(x,s)\geq k_1(x)|s|^{\nu}-k_2(x)-k_3(x)|s|^{1-\beta}$, 
condition \eqref{diffeoasymptoticbb} yields
    \begin{equation*}
    G^\sharp (x,s)\geq k_1(x)|\varphi(s)|^{\nu}-k_2(x)-k_3(x)|\varphi(s)|^{1-\beta}
     \geq  k^\sharp_1(x)|s|^{\nu^\sharp}-k^\sharp_2(x)- k^\sharp_3(x)|s|,
\end{equation*}
    for a.e.~$x\in\Omega$ and all $s\in\R$, for suitable 
$k_j^\sharp:\Omega\to\R$, $j=1,2,3$, with the stated summability.
\end{proof}

\noindent
Now, we see how the nonlinearity $g$ gets modified under the action of a diffeomorphism. 

\begin{proposition}
    \label{newgrow}
	Let $\varphi\in C^2(\R)$ be a diffeomorphism which satisfies 
	the properties of Definition~\ref{diffeoclass} with $0\leq\beta<2/N$, $N\geq 3$ and such that
	\eqref{diffeoasymptoticbb} holds. Let $g:\Omega\times\R\to\R$ satisfy
\begin{equation}
    \label{growthassty}
|g(x,s)|\leq a(x)+b|s|^{p-1}\qquad\text{for a.e. $x\in\Omega$ and
all $s\in\R$},
\end{equation}
for some $a\in L^{q+\beta q(p-1)^{-1}}(\Omega)$, $q\geq \frac{2N}{N+2}$, $b\geq 0$ with $2<p\leq 2^*(1-\beta)$.
Then, we have
\begin{equation*}
|g^\sharp(x,s)|\leq
a^\sharp(x)+b|s|^{p^\sharp-1}\quad\text{for a.e.
$x\in\Omega$ and all $s\in\R$},
\end{equation*}
for some $2<p^\sharp \leq 2^*$ and $a^\sharp\in L^{q}(\Omega)$.
\end{proposition}
\begin{proof}
Taking into account~\eqref{diffeoasymptoticbb} and \eqref{two}, for a.e.~$x\in\Omega$ and all $s\in\R$ we have 
\begin{equation*}
|g^\sharp(x,s)| \leq a(x)\varphi'(s)+b|\varphi(s)|^{p-1}\varphi'(s)  
\leq Ca(x)+C+Ca(x)^{\frac{p+\beta-1}{p-1}}+C|s|^{\frac{p}{1-\beta}-1},
\end{equation*}
yielding the assertion with $p^\sharp=\frac{p}{1-\beta}$ and $a^\sharp=Ca+C+Ca^{\frac{p+\beta-1}{p-1}}$.
\end{proof}

\subsection{Sign conditions}
\label{signsect}

The classical sign condition \eqref{classicsign} is {\em not} invariant under diffeomorphism
as Proposition \ref{invspec} shows. The next proposition introduces a different kind of sign condition 
that remains invariant under the effect of $\varphi$.

\begin{proposition}
    \label{invar1}
	Let $\varphi\in C^2(\R)$ be a diffeomorphism which satisfies 
	the properties of Definition~\ref{diffeoclass}.
Assume that there exist $\eps\in (0,1-\beta]$ and $R\geq 0$ such that
\begin{equation}
	\label{generalsignn}
(1-\eps) j_\xi(x,s,\xi)\cdot\xi+j_s(x,s,\xi)s\geq 0,
\end{equation}
for a.e.~$x\in\Omega$ and all $(s,\xi)\in\R\times\R^N$ with
$|s|\geq R$. \vskip4pt \noindent Then there exist $\eps^\sharp\in
(0,1]$ and $R^\sharp>0$ such that
$$
(1-\eps^\sharp) j
^\sharp_\xi(x,s,\xi)\cdot\xi+j^\sharp_s(x,s,\xi)s\geq 0,
$$
for a.e.~$x\in\Omega$ and all $(s,\xi)\in\R\times\R^N$ with
$|s|\geq R^\sharp$.
\end{proposition}
\begin{proof}
    Let us write $\eps=\eps_0(1-\beta)$, for some $\eps_0\in (0,1]$.
By taking into account~\eqref{two}, there exists $0<\delta<\eps_0 (1+\eps_0(1-\beta))^{-1}$ 
and $R^\sharp>0$ sufficiently large that
$$
1+\frac{\varphi''(s)s}{\varphi'(s)}\geq \frac{\varphi'(s)s}{\varphi(s)}-\delta,
\qquad
\frac{\varphi'(s)s}{\varphi(s)}\geq \frac{1}{1-\beta}-\delta,
$$
and $|\varphi(s)|\geq R$ for all $s\in\R$ such that $|s|\geq R^\sharp $. Then, in turn, we get
\begin{align*}
    & j^\sharp_\xi(x,s,\xi)\cdot\xi+j^\sharp_s(x,s,\xi)s \\
    &=\Big(1+\frac{\varphi''(s)s}{\varphi'(s)}\Big)j_\xi(x,\varphi(s),\varphi'(s)\xi)\cdot \varphi'(s)\xi
    +\frac{\varphi'(s)s}{\varphi(s)} j_s(x,\varphi(s),\varphi'(s)\xi)\varphi(s) \\
    &\geq \frac{\varphi'(s)s}{\varphi(s)}\big(j_\xi(x,\varphi(s),\varphi'(s)\xi)\cdot \varphi'(s)\xi
    +j_s(x,\varphi(s),\varphi'(s)\xi)\varphi(s)\big) \\
    \noalign{\vskip4pt}
&\quad  -\delta j_\xi(x,\varphi(s),\varphi'(s)\xi)\cdot \varphi'(s)\xi,
\end{align*}
for a.e.~$x\in\Omega$ and all $(s,\xi)\in\R\times\R^N$ with
$|s|\geq R^\sharp$. Setting
$$
\eps^\sharp =\eps_0-\delta(1+\eps_0(1-\beta))\in (0,1],
$$
it follows by assumption that
\begin{equation*}
     j^\sharp_\xi(x,s,\xi)\cdot\xi+j^\sharp_s(x,s,\xi)s  
\geq \Big(\eps\frac{\varphi'(s)s}{\varphi(s)}-\delta\Big) j_\xi(x,\varphi(s),\varphi'(s)\xi)\cdot \varphi'(s)\xi
     \geq  \eps^\sharp  j^\sharp_\xi(x,s,\xi)\cdot\xi ,
\end{equation*}
for a.e.~$x\in\Omega$ and all $(s,\xi)\in\R\times\R^N$ with
$|s|\geq R^\sharp$. This concludes the proof.
\end{proof}

\begin{remark}\rm
In the literature of quasi-linear problems like \eqref{equation-intro} the (say, positive) sign condition $j_s(x,s,\xi)s\geq 0$
is a classical assumption (cf.\ \cite{Ab,candeg} and references therein), helping to achieve both existence and summability of 
the solutions. On the other hand, in \cite{pellacci-nos}, when $j(x,s,\xi)=A(x,s)\xi\cdot\xi$, 
the existence of solutions is obtained either with the opposite sign 
condition or even without any sign hypothesis at all. To handle this situation, alternative conditions as \cite[Assumption 1.5]{pellacci-nos}
are assumed, which imply \eqref{generalsignn} (at least for $s\geq R$) for suitable $\eps$, as it can be easily verified.
\end{remark}

Under the generalized sign condition \eqref{generalsignn}, we get a summability result
which improves \cite[Lemma 4.6]{pelsqu}.
This also shows that condition \eqref{summab-solger} in Definition \ref{defsol-bis} is natural.
For a function $f$, the notation $|df|(u)$ stands for the weak slope of $f$ at $u$ (cf.~e.g.\ \cite{CDM,DM}).

\begin{proposition}\label{sommab}
Assume that \eqref{originalgrowths2} holds and that 
there exist $\eps\in (0,1)$ and $R\geq 0$ with
\begin{equation}
    \label{gensign}
(1-\eps) j_\xi(x,s,\xi)\cdot\xi+j_s(x,s,\xi)s\geq 0,
\end{equation}
for a.e.~$x\in\Omega$ and all $(s,\xi)\in\R\times\R^N$ with $|s|\geq R$. Let us set
$$
I(u)=\int_\Omega j(x,u,\nabla u),\quad u\in H^1_0(\Omega).
$$
Then, for every $u\in {\rm dom}(I)$ with $|dI|(u)<+\infty$, we have
\begin{align}
\label{stimaslope} \int_\Omega j_\xi(x,u,\nabla u)\cdot\nabla u +
j_s(x,u,\nabla u)u \leq |dI|(u)\|u\|_{1,2}.
\end{align}
In particular, there holds
$$
j_\xi(x,u,\nabla u)\cdot \nabla u\in L^1(\Omega), \qquad 
j_s(x,u,\nabla u)u\in L^1(\Omega),
$$
and there exists $\Psi\in H^{-1}(\Omega)$ with $\|\Psi\|_{H^{-1}}\leq |dI|(u)$ such that 
\begin{equation*}
\into \jxi{u}\cdot \nabla w +\into
\js{u}w=\langle \Psi,w\rangle, \quad\forall w\in V_u.
\end{equation*}
\end{proposition}
\begin{proof}
Let $b\in\R$ be such that $b>I(u)$. Notice first that if $u$ is such that 
$$
\int_\Omega j_\xi(x,u,\nabla u)\cdot \nabla u+j_s(x,u,\nabla u)u\leq 0,
$$
then the conclusion holds. Otherwise, let $\sigma$ be an arbitrary positive number such that
$$
\int_\Omega\,j_\xi(x,u,\nabla u)\cdot\nabla u\,+ j_s(x,u,\nabla u)u>\sigma\|u\|_{1,2}.
$$
Fixed $\eta>0$, we set $\alpha^{-1}=\|u\|_{1,2}(1+\eta)$.
Let us prove that there exist $\delta>0$ such that, for all $v\in B(u,\delta)$ and for any
$\tau \in L^{\infty}(\Omega)$ with $\|\tau\|_{\infty}<\delta$, it follows
\begin{equation}
    \label{qdist1}
\int_\Omega [j_s(x,w,(1-\alpha \tau)\nabla v)v+ j_\xi(x,w,(1-\alpha \tau)\nabla v)\cdot\nabla v]>\sigma \|u\|_{1,2},
\end{equation}
where $w=(1-\alpha \tau)v$. In fact, assume by contradiction that this is not the case. Then, we
find a sequence $(v_n) \subset \hsob$ with $\|v_n-u\|_{1,2}\to 0$ as $n\to\infty$ and a sequence
$(\tau_n) \subset L^{\infty}(\Omega)$ with $\|\tau_n\|_{\infty}\to 0$ as $n\to\infty$
such that, denoting $w_n=(1-\alpha\tau_n)v_n$ for all $n\geq 1$, it holds
\begin{equation}
    \label{contradarg}
\int_\Omega [j_s(x,w_n,(1-\alpha \tau_n)\nabla v_n)v_n+ j_\xi(x,w_n,(1-\alpha \tau_n)\nabla v_n)\cdot\nabla v_n]\leq\sigma \|u\|_{1,2}.
\end{equation}
Since $v_n \to u$ in $\hsob$ and $\tau_n\to 0$ in $L^{\infty}(\Omega)$ as $n\to\infty$, a.e.~in $\Omega$
we have that
$$
j_s(x,w_n,(1-\alpha \tau_n)\nabla v_n)v_n+ j_\xi(x,w_n,(1-\alpha \tau_n)\nabla v_n)\cdot \nabla v_n
\to
j_s(x,u,\nabla u)u+j_\xi(x,u,\nabla u)\cdot \nabla u.
$$
Moreover there exists a positive constant $C(R)$ such that, for every $n\geq 1$,
\begin{equation}
	\label{4fatouprep}
j_s(x,w_n,(1-\alpha \tau_n)\nabla v_n)v_n+ j_\xi(x,w_n,(1-\alpha \tau_n)\nabla v_n)\cdot\nabla v_n\geq -C(R)|\nabla v_n|^2.
\end{equation}
In fact, if $|w_n(x)|\geq R$, from condition~\eqref{gensign} the left hand side is nonnegative.
If instead $|w_n(x)|\leq R$, we can assume $|v_n(x)|\leq 2R$, and by \eqref{originalgrowths2} we get
\begin{align*}
 |j_s(x,w_n,&(1-\alpha \tau_n)\nabla v_n)v_n+ j_\xi(x,w_n,(1-\alpha \tau_n)\nabla v_n)\cdot\nabla v_n| \\
& \leq \gamma(|w_n|)|v_n||\nabla v_n|^2+\mu(|w_n|)|\nabla v_n|^2\leq (2\gamma(R)R+\mu(R))|\nabla v_n|^2.
\end{align*}
Then, we are allowed to apply Fatou's Lemma, yielding
\begin{gather*}
\liminf_{n\to \infty}
\int_\Omega [j_s(x,w_n,(1-\alpha \tau_n)\nabla v_n)v_n+ j_\xi(x,w_n,(1-\alpha \tau_n)\nabla v_n)\cdot \nabla v_n]
\\
\geq\int_\Omega j_s(x,u,\nabla u)u \,+j_\xi(x,u,\nabla
u)\cdot\nabla u >\sigma \|u\|_{1,2},
\end{gather*}
which immediately yields a contradiction with \eqref{contradarg}. Hence~\eqref{qdist1} holds,
for some $\delta>0$. Observe that,
since $j(x,\cdot,\cdot)$ is of class $C^1$ for a.e. $x \in \Omega$ then, for any $t \in [0,1]$ and
every $v\in {\rm dom}(I)$, there exists $0\leq \tau(x,t)\leq t$ such that
\begin{align}
	\label{identlagrange}
&j(x,(1-\alpha t)v,(1-\alpha t)\nabla v)- j(x,v,\nabla v)=\\
&-\alpha t [j_s(x,(1-\alpha \tau)v,(1-\alpha \tau)\nabla v)v+ j_\xi(x,(1-\alpha \tau)v,(1-\alpha \tau)\nabla v)\cdot \nabla v].
\notag
\end{align}
As for the inequality \eqref{4fatouprep}, for some $C(R)>0$, for $t$ small enough it holds
\begin{equation*}
j_s(x,(1-\alpha \tau)v,(1-\alpha \tau)\nabla v)v+ j_\xi(x,(1-\alpha \tau)v,(1-\alpha \tau)\nabla v)\cdot\nabla v\geq -C(R)|\nabla v|^2.
\end{equation*}
Whence, if $v\in {\rm dom}(I)$ by \eqref{identlagrange} it follows that $(1-\alpha t)v\in {\rm dom}(I)$ for all $t\in [0,\delta]$ and
\begin{equation}
    \label{summbbbb}
j_s(x,(1-\alpha \tau)v,(1-\alpha \tau)\nabla v)v+ j_\xi(x,(1-\alpha \tau)v,(1-\alpha \tau)\nabla v)\cdot\nabla v\in L^1(\Omega).
\end{equation}
Up to reducing $\delta$, we may assume that $\delta<\eta
\|u\|_{1,2}$. Then, for all $v \in B(u,\delta)$, we have $\|v\|_{1,2} \leq (1+\eta)\|u\|_{1,2}=\alpha^{-1}$.
Consider the continuous map ${\mathcal H}:B(u,\delta)\cap I^b\times [0,\delta]\to \hsob$ defined as
${\mathcal H}(v,t)=(1-\alpha t)v$, where $I^b=\{v\in H^1_0(\Omega):I(v)\leq b\}$.
From \eqref{qdist1} (applied, for each $t\in [0,\delta]$, with the function
$\tau(\cdot,t)\in L^\infty(\Omega,[0,\delta])$ for which identity~\eqref{identlagrange} holds)
and identity \eqref{identlagrange}, for every $t\in [0,\delta]$ and $v\in B(u,\delta)\cap I^b$ we have
\begin{equation*}
\|{\mathcal H}(v,t)-v\|_{1,2}\leq t,
\qquad
I({\mathcal H}(v,t))\leq I(v) - \frac{ \sigma }{1+ \eta}t.
\end{equation*}   
Then, by means of \cite[Proposition 2.5]{DM} and exploiting the arbitrariness of $\eta$, we get
$|dI|(u)\geq\sigma$. In turn, \eqref{stimaslope} follows from the
arbitrariness of $\sigma$. Concerning the second part of the statement, since $|dI|(u)<+\infty$,
from~\eqref{gensign} and~\eqref{stimaslope},
\begin{equation}
    \label{summbconclus}
j_\xi(x,u,\nabla u)\cdot\nabla u +
j_s(x,u,\nabla u)u\in L^1(\Omega).
\end{equation}
In turn, using again~\eqref{gensign}, it follows $j_\xi(x,u,\nabla u)\cdot\nabla u\in L^1(\Omega)$, since
\begin{align*}
\eps j_\xi(x,u,\nabla u)\cdot\nabla u &\leq
\eps\mu(R)|\nabla u|^2+\eps j_\xi(x,u,\nabla u)\cdot\nabla u\chi_{\{|u|\geq R\}} \\
& \leq  \eps\mu(R)|\nabla u|^2+|j_s(x,u,\nabla u)u+j_\xi(x,u,\nabla u)\cdot\nabla u|.
\end{align*}
Then, by exploiting \eqref{summbconclus} again, $j_s(x,u,\nabla u)u\in L^1(\Omega)$.
The final assertion does not rely upon any sign condition and follows directly 
from \cite[Proposition 4.5]{pelsqu}. This concludes the proof.
\end{proof}

In the next result we show that it is possible to enlarge the
class of admissible test functions. In order to do this, suppose
we have a function $u\in\hsob$  such that
\begin{equation}\label{eqb}
 \int_{\Omega}j_\xi(x,u,\nabla u)\cdot \nabla z+
\into j_s(x,u,\nabla u) z=\langle w,z\rangle, \qquad \forall z\in
V_u,
\end{equation}
for $w\in \dhsob$. Under suitable assumptions, if \eqref{gensign} holds true,
we can use $\zeta u\in\hsob$ with $\zeta\in L^\infty(\Omega)$
as an admissible test functions in~\eqref{eqb}, generalizing \cite[Theorem 4.8]{pelsqu}.

\begin{proposition}
	\label{BB}
Assume that \eqref{originalgrowths2} and \eqref{gensign} hold.
Let $w\in H^{-1}(\Omega)$, and let $u\in
H^1_0(\Omega)$ be such that~\eqref{eqb} is satisfied. Moreover,
suppose that $j_\xi(x,u,\nabla u)\cdot\nabla u\in L^1(\Omega)$
and that there exist  $v\in H^1_0(\Omega)$ and $\eta\in
L^1(\Omega)$ such that
\begin{equation}\label{control2}
j_s(x,u,\nabla u)v\geq\eta \qquad \text{and} \qquad
j_\xi(x,u,\nabla u)\cdot\nabla v\geq\eta.
\end{equation}
Then $j_s(x,u,\nabla u)v\in L^1(\Omega)$, $j_\xi(x,u,\nabla
u)\cdot\nabla v\in L^1(\Omega)$ and
\begin{equation}\label{eqb2}
\int_{\Omega} j_\xi(x,u,\nabla u)\cdot\nabla v +\into
j_s(x,u,\nabla u)v =\langle w,v\rangle.
\end{equation}
In particular, if $\zeta\in L^\infty(\Omega)$, $\zeta\geq 0$, 
$\zeta u\in H^1_0(\Omega)$ and $j_\xi(x,u,\nabla u)\cdot\nabla (\zeta u)\in L^1(\Omega)$
then it follows that $j_s(x,u,\nabla u)\zeta u\in L^1(\Omega)$ and
\begin{equation}
	\label{testu}
\int_{\Omega}  j_\xi(x,u,\nabla u)\cdot\nabla (\zeta u) +\into
j_s(x,u,\nabla u)\zeta u =\langle w,\zeta u\rangle.
\end{equation}
\end{proposition}
\begin{proof}
The first part of the statement follows by means of \cite[Theorem 4.8]{pelsqu}. By assumption \eqref{gensign} 
and since $\zeta$ is nonnegative and bounded, we have
\begin{align*}
j_s(x,u,\nabla u)\zeta u &=\zeta j_s(x,u,\nabla u)u\chi_{\{|u|\leq R\}}+\zeta j_s(x,u,\nabla u)u\chi_{\{|u|\geq R\}} \\
& \geq -R\gamma(R)\|\zeta\|_{L^\infty(\Omega)}|\nabla u|^2-(1-\eps) \zeta j_\xi(x,u,\nabla u)\cdot\nabla u \in L^1(\Omega).
\end{align*}
The last assertion of the statement then follows from the first one.
\end{proof}

\subsection{AR type conditions}
\label{ambrrabsect}
Some Ambrosetti-Rabinowitz type conditions, typically used in order 
to guarantee the boundedness of Palais-Smale sequences, remain invariant.

\begin{proposition}
    \label{invar2}
	Let $\varphi\in C^2(\R)$ be a diffeomorphism which satisfies 
	the properties of Definition~\ref{diffeoclass}.
Assume that there exist $\delta>0$, $\nu>2(1-\beta)$ and $R\geq 0$ such that
$$
\nu j(x,s,\xi)-(1+\delta)j_\xi(x,s,\xi)\cdot\xi-j_s(x,s,\xi)s-\nu G(x,s)+g(x,s)s\geq 0,
$$
and $G(x,s)\geq 0$ for a.e.~$x\in\Omega$ and all $(s,\xi)\in\R\times\R^N$ with $|s|\geq R$. 
\vskip4pt 
\noindent 
Then there exist $\delta^\sharp>0$, $\nu^\sharp>2$ and $R^\sharp>0$ such that
$$
\nu^\sharp j^\sharp(x,s,\xi)-(1+\delta^\sharp )
j^\sharp_\xi(x,s,\xi)\cdot\xi-j^\sharp_s(x,s,\xi)s- \nu^\sharp
G^\sharp(x,s)+g^\sharp (x,s)s \geq 0,
$$
and $G^\sharp(x,s)\geq 0$ for a.e.~$x\in\Omega$ and all $(s,\xi)\in\R\times\R^N$ with $|s|\geq R^\sharp$.
\end{proposition}
\begin{proof}
A direct calculation yields
\begin{align*}
&   \frac{\nu}{1-\beta}
j^\sharp(x,s,\xi)-j^\sharp_\xi(x,s,\xi)\cdot\xi
-j^\sharp_s(x,s,\xi)s-\frac{\nu}{1-\beta}  G^\sharp(x,s)+ g^\sharp(x,s)s \\
&=\frac{\nu}{1-\beta} j(x,\varphi(s),\varphi'(s)\xi)-
\Big(1+\frac{\varphi''(s)s}{\varphi'(s)}\Big)j_\xi(x,\varphi(s),\varphi'(s)\xi)\cdot \varphi'(s)\xi  \\
& -\frac{\varphi'(s)s}{\varphi(s)} j_s(x,\varphi(s),\varphi'(s)\xi)\varphi(s)-
\frac{\nu}{1-\beta} G(x,\varphi(s))+\frac{\varphi'(s)s}{\varphi(s)} g(x,\varphi(s))\varphi(s) \\
&= \frac{\varphi'(s)s}{\varphi(s)}\Big( \frac{\varphi(s)}{\varphi'(s)s}\frac{\nu}{1-\beta} j(x,\varphi(s),\varphi'(s)\xi) \\
\noalign{\vskip3pt}
& -\frac{\varphi(s)}{\varphi'(s)s}\Big(1+\frac{\varphi''(s)s}{\varphi'(s)}\Big)j_\xi(x,\varphi(s),\varphi'(s)\xi)\cdot \varphi'(s)\xi  \\
& -j_s(x,\varphi(s),\varphi'(s)\xi)\varphi(s)-
\frac{\nu}{1-\beta}\frac{\varphi(s)}{\varphi'(s)s} G(x,\varphi(s))+g(x,\varphi(s))\varphi(s) \Big),
\end{align*}
for a.e.~$x\in\Omega$ and all $(s,\xi)\in\R\times\R^N$ such that $s\neq 0$.
We recall that $j(x,\tau,\zeta)\geq 0$, $j_\xi(x,\tau,\zeta)\cdot \zeta\geq 0$ and that the map $s\mapsto s\varphi(s)$ is nonnegative.
Therefore, on account of condition \eqref{two}, for all $\eta>0$ small enough there exists
$R^\sharp>0$ large enough that $|\varphi(s)|\geq R$ for all $s\in\R$ with $|s|\geq R^\sharp$ and 
\begin{align*}
&   \frac{\nu}{1-\beta}
j^\sharp(x,s,\xi)-j^\sharp_\xi(x,s,\xi)\cdot\xi
-j^\sharp_s(x,s,\xi)s-\frac{\nu}{1-\beta} G^\sharp (x,s)+g^\sharp (x,s)s \\
&\geq \frac{\varphi'(s)s}{\varphi(s)} \Big(\nu j(x,\varphi(s),\varphi'(s)\xi)-\eta(1-\beta)j(x,\varphi(s),\varphi'(s)\xi)  \\
\noalign{\vskip2.5pt}
&- j_\xi(x,\varphi(s),\varphi'(s)\xi)\cdot \varphi'(s)\xi
-\eta(1-\beta) j_\xi(x,\varphi(s),\varphi'(s)\xi)\cdot \varphi'(s)\xi  \\
\noalign{\vskip2.5pt}
& -j_s(x,\varphi(s),\varphi'(s)\xi)\varphi(s)-
\nu G(x,\varphi(s))-\eta(1-\beta) G(x,\varphi(s)) +g(x,\varphi(s))\varphi(s)\Big)\\
\noalign{\vskip2.5pt}
& \geq ((1-\beta)^{-1}-\eta)(\delta-\eta(1-\beta)) j^\sharp_\xi(x,s,\xi)\cdot \xi \\
\noalign{\vskip2.5pt} &
-\frac{\varphi'(s)s}{\varphi(s)}(1-\beta)\eta j^\sharp(x,s,\xi)
 -\frac{\varphi'(s)s}{\varphi(s)}(1-\beta)\eta G^\sharp (x,s) \\
\noalign{\vskip2pt} &\geq
((1-\beta)^{-1}-\eta)(\delta-\eta(1-\beta))
j^\sharp_\xi(x,s,\xi)\cdot \xi -2\eta j^\sharp(x,s,\xi)
 -2\eta G^\sharp(x,s),
\end{align*}
for a.e.~$x\in\Omega$ and all $(s,\xi)\in\R\times\R^N$ such that $|s|\geq R^\sharp$.
Finally, since by convexity of $j^\sharp$ and $j^\sharp(x,s,0)=0$ we have
$j^\sharp_\xi(x,s,\xi)\cdot\xi\geq j^\sharp(x,s,\xi)$, we get
\begin{gather*}
\frac{\nu}{1-\beta}
j^\sharp(x,s,\xi)-j^\sharp_\xi(x,s,\xi)\cdot\xi
-j^\sharp_s(x,s,\xi)s-\frac{\nu}{1-\beta} G^\sharp(x,s)+g^\sharp(x,s)s  \\
\geq \delta^\sharp j^\sharp_\xi(x,s,\xi)\cdot \xi+2\eta
j^\sharp(x,s,\xi) -2\eta G^\sharp(x,s).
\end{gather*}
In turn, choosing $\eta$ small enough and setting
$$
\delta^\sharp=(1-\beta)^{-1}\delta-\eta(5+\delta)+\eta^2(1-\beta)>0,\qquad
\nu^\sharp =\nu(1-\beta)^{-1}-2\eta>2,
$$
the assertion follows.
\end{proof}

\begin{corollary}
    \label{invar2-cor}
	Let $\varphi\in C^2(\R)$ be a diffeomorphism satisfying
	the properties of Definition \ref{diffeoclass}.
Assume that $\xi\mapsto j(x,s,\xi)$ is homogeneous of degree two and that
there are $\nu>2$ and $R>0$ with
\begin{equation}
    \label{separatasss}
j_s(x,s,\xi)s\leq 0,\qquad 0\leq \nu G(x,s)\leq g(x,s)s,
\end{equation}
for a.e.~$x\in\Omega$ and all $(s,\xi)\in\R\times\R^N$ with $|s|\geq R$. Then
$$
\nu^\sharp  j^\sharp(x,s,\xi)-(1+
\delta^\sharp)j^\sharp_\xi(x,s,\xi)\cdot\xi-j^\sharp_s(x,s,\xi)s-
\nu^\sharp G^\sharp(x,s)+g^\sharp(x,s)s \geq 0,
$$
for a.e.~$x\in\Omega$ and all $(s,\xi)\in\R\times\R^N$ with
$|s|\geq R^\sharp$, for some $\delta^\sharp>0$, $R^\sharp>0$ and $\nu^\sharp>2$.
\end{corollary}
\begin{proof}
Since $\xi\mapsto j(x,s,\xi)$ is $2$-homogeneous and $\nu>2$, there exists $\delta>0$ with
$$
\nu j(x,s,\xi)-(1+\delta)j_\xi(x,s,\xi)\cdot\xi=(\nu-2-2\delta)j(x,s,\xi)\geq 0,
$$
for a.e.~$x\in\Omega$ and all $(s,\xi)\in\R\times\R^N$. Hence, by assumptions~\eqref{separatasss}, we get
$$
\nu j(x,s,\xi)-(1+\delta)j_\xi(x,s,\xi)\cdot\xi-j_s(x,s,\xi)s-\nu G(x,s)+g(x,s)s\geq 0,
$$
for a.e.~$x\in\Omega$ and all $(s,\xi)\in\R\times\R^N$ with $|s|\geq R$.
Proposition~\ref{invar2} yields the assertion.
\end{proof}

\section{Multiplicity of solutions}
\label{existence-sect}
\noindent
As a by-product of the previous results, we obtain the following existence result.
Compared with the results of \cite{AbO} here we can get infinitely many solution, not necessarily 
bounded.

\begin{theorem}
	\label{existthmm}
	Assume that $\varphi\in C^2(\R)$ satisfies the properties of Definition~\ref{diffeoclass},
\eqref{diffeoasymptoticbb} and let $N\geq 3$. Moreover, 
let $j:\Omega\times\R\times\R^N\to\R$ satisfy \eqref{originalgrowths1}-\eqref{originalgrowths2},
$\xi\mapsto j(x,s,\xi)$ be strictly convex, and
\begin{align}
	\label{parita}
j(x,-s-\xi)=j(x,s,\xi),\quad\text{for a.e.\ $x\in\Omega$ and all $(s,\xi)\in\R\times\R^N$,} \\ 
	\label{signostrongg}
j_s^\sharp(x,s,\xi)s\geq 0,\quad \text{for all $|s|\geq R^\sharp$ and some $R^\sharp\geq 0$}.
\end{align}
Let $g:\Omega\times\R\to\R$ be continuous, satisfying \eqref{growthassty} with $2<p<2^*(1-\beta)$,
\begin{equation}
	\label{g-disparit}
g(x,-s)=-g(x,s),\quad\text{for a.e.\ $x\in\Omega$ and all $s\in\R$,} 
\end{equation}
$G(x,s)\geq 0$ for $|s|\geq R$ and the 
joint conditions \eqref{genARcondition} and \eqref{cdpdmm}, for some $R\geq 0$. Then, 
\begin{equation*}
\begin{cases}
-\dvg( j_\xi(x,u,\nabla u)) +j_s(x,u,\nabla u)=g(x,u), & \text{in $\Omega$}, \\
\quad u=0, & \text{on $\partial\Omega$}
\end{cases}
\end{equation*}
admits a sequence $(u_n)$ of generalized solutions in the sense of Definition~\ref{defsol-bis}. Furthermore,
\begin{align*}
\frac{2N}{N+2}<q<\frac{N}{2}
&\quad\Longrightarrow\quad
u_n\in L^{\frac{Nq(1-\beta)}{N-2q}}(\Omega),  \\
q>\frac{N}{2}  &\quad\Longrightarrow\quad u_n\in
L^{\infty}(\Omega),
\end{align*}
in the notations of assumptions \eqref{growthassty}. In particular, if $q>N/2$,
it follows that $u_h\in H^1_0(\Omega)\cap L^{\infty}(\Omega)$ are solutions in distributional sense.
\end{theorem}
\begin{proof}
Of course, $\xi\mapsto j^\sharp(x,s,\xi)$ is strictly convex. 
By assumptions \eqref{originalgrowths1}-\eqref{originalgrowths2}, \eqref{growthassty}, \eqref{genARcondition} and \eqref{cdpdmm},
in light of Propositions~\ref{rm1}, \ref{menoinf}, \ref{newgrow} and \ref{invar2} and taking into account the sign condition
\eqref{signostrongg} for $j^\sharp$, \cite[assumptions (1.1)-(1.4), (1.7), (2.2), (2.4) and the variant 
\eqref{genARcondition} for $j^\sharp$ of conditions (1.9) and (2.3)
joined together which still guarantees the boundedness of Palais-Smale sequences]{pelsqu} 
are satisfied for $j^\sharp$ and $g^\sharp$ for some $R^\sharp$. 
Also, since $\varphi$ is odd, \eqref{parita} yields
$$
j^\sharp(x,-s,-\xi)=j(x,\varphi(-s),-\varphi'(-s)\xi)=j(x,-\varphi(s),-\varphi'(s)\xi)=j^\sharp(x,s,\xi),
$$
for a.e.\ $x\in\Omega$ and all $(s,\xi)\in\R\times\R^N$ and, analogously, \eqref{g-disparit} yields
$$
g^\sharp(x,-s)=g(x,\varphi(-s))\varphi'(-s)=g(x,-\varphi(s))\varphi'(s)=-g^\sharp(x,s),
$$
for a.e.\ $x\in\Omega$ and all $s\in\R$. 
Then, we are allowed to apply \cite[Theorem 2.1]{pelsqu} and obtain a sequence $(v_h)\subset H^1_0(\Omega)$
of generalized solutions of \eqref{probmoddd} in the sense of \cite{pelsqu}, namely 
\begin{equation*}
j_\xi^\sharp(x,v_h,\nabla v_h)\cdot\nabla v_h\in\elle1,\qquad j_s^\sharp(x,v_h,\nabla v_h)v_h  \in \elle1,
\end{equation*}
and
\begin{equation*}
\into j_\xi^\sharp(x,v_h,\nabla v_h)\cdot\nabla \psi +\into
j_s^\sharp(x,v_h,\nabla v_h)\psi=\into g^\sharp(x,v_h)\psi, \quad\forall \psi\in V_{v_h}.
\end{equation*}
In particular, $(v_n)$ is a sequence of $H^1_0(\Omega)$ generalized solutions of problem \eqref{probmoddd}
in the sense of Definition~\ref{defsol-bis}.
The desired existence assertion now follows from Proposition~\ref{soluzdisrt} for $u_n=\varphi(v_n)$.
Concerning the summability, if $a^\sharp\in L^r(\Omega)$ and	
$|g^\sharp(x,s)|\leq a^\sharp(x)+b|s|^{(N+2)/(N-2)}$ for a.e. $x\in\Omega$ and all $s\in\R$,
then, by \cite[Theorem 7.1]{pelsqu}, a generalized solution $v\in H^1_0(\Omega)$ of problem \eqref{probmoddd} belongs to
$L^{Nr/(N-2r)}(\Omega)$ for any $2N/(N+2)<r<N/2$ and to $L^\infty(\Omega)$, for all $r>N/2$. 
Since $g$ is subjected to \eqref{growthassty}, by Proposition~\ref{newgrow}, we also get the final conclusions.
\end{proof}

\begin{remark}\rm
We believe that Theorem~\ref{existthmm} remains true if \eqref{signostrongg} is substituted by \eqref{generalsign}.
\end{remark}

\begin{remark}\rm
For $\beta=0$, the summability of solutions coincide with the standard one.
\end{remark}

\noindent
The next proposition yields a class of $j$, which is the one studied in \cite{AbO} (condition \eqref{Ab-13}
below is precisely condition (1.3) in \cite{AbO}), 
satisfying the assumptions of Theorem~\ref{existthmm}.

\begin{proposition}
    \label{invspec}
    Assume that $j:\Omega\times\R\times\R^N\to\R$ is of the form
    $$
j(x,s,\xi)=\frac{1}{2}a(x,s)|\xi|^2,
    $$
    where $a(x,\cdot)\in C^1(\R,\R^+)$ for a.e.\ $x\in \Omega$. Assume furthermore that there exist $R\geq 0$ such that
\begin{equation}
	\label{Ab-13}
-2\beta a(x,s)  \leq D_sa(x,s)(1+|s|){\rm sign}(s) \leq 0,
\end{equation}
for a.e.\ $x \in \Omega$ and all $s\in\R$ with $|s|\geq R$.
Let $\varphi\in C^2(\R)$ be a diffeomorphism according to Definition~\ref{diffeoclass}
which is addition satisfies 
\begin{equation} 
	\label{four}
\varphi''(s)-\frac{\beta\varphi'(s)^2}{1+\varphi(s)}\geq 0,\qquad\text{for all
$s\in\R$ with $s\geq 1$}.
\end{equation}
Then there exist $\nu^\sharp>2$, $\delta^\sharp>0$ and $R^\sharp>0$ such that
\begin{equation*}
    sj^\sharp_s(x,s,\xi)\geq 0,\qquad
    \nu^\sharp j^\sharp(x,s,\xi)-(1+\delta^\sharp)j^\sharp_\xi(x,s,\xi)\cdot\xi
    -j^\sharp_s(x,s,\xi)s\geq 0
\end{equation*}
for a.e.~$x \in \Omega$, all $\xi \in {\R}^N$, and every $s \in
\R$ with $|s|\geq R^\sharp$.
\end{proposition}

\begin{proof}
Let $R^\sharp\geq 1$ be such that
$|\varphi(s)|\geq R$ for all $s\in\R$ with $|s|\geq R^\sharp$. Then, by \eqref{Ab-13},
for all $s\geq R^\sharp$ we have $\varphi(s)\geq R$ and 
\begin{equation*}
\begin{aligned}
j^\sharp_s(x,s,\xi) &= [D_sa(x,\varphi(s))
(\varphi'(s))^3+2\varphi'(s)\varphi''(s) a(x,\varphi(s))]|\xi|^2/2  \\
&\geq a(x,\varphi(s)) \varphi'(s)\Big[\frac{-\beta
\varphi'(s)^2}{1+\varphi(s)}+\varphi''(s)\Big]|\xi|^2.
\end{aligned}
\end{equation*}
Recalling that $a(x,\varphi(s))$ and $\varphi'(s)$ are positive
and by \eqref{four}, one gets $j^\sharp_s(x,s,\xi)\geq 0$. 
Similarly, if $s\leq-R^\sharp$, again by \eqref{Ab-13}, we have $\varphi(s)\leq -R$ and 
\begin{equation*}
j^\sharp_s(x,s,\xi)\leq a(x,\varphi(s))
\varphi'(s)\Big[\frac{\beta \varphi'(s)^2}{1+|\varphi(s)|}+\varphi''(s)\Big]|\xi|^2,
\end{equation*}
and so that $j^\sharp_s(x,s,\xi)\leq 0$,  again due to \eqref{four}, since being
$\varphi$ and $\varphi''$ odd and $\varphi'$ even yields
\begin{equation*} 
\varphi''(s)+\frac{\beta\varphi'(s)^2}{1+|\varphi(s)|}\leq 0,\qquad\text{for all
$s\in\R$ with $s\leq -1$}.
\end{equation*}
The second inequality in the assertion follows from
Corollary \ref{invar2-cor} (applied with $g=0$), since $\xi\mapsto j(x,s,\xi)$ is $2$-homogeneous 
and $j_s(x,s,\xi)s\leq 0$ for a.e. $x \in \Omega$, all $\xi \in \R^N$ 
and any $|s|\geq R$.
\end{proof}

\begin{remark}\rm
	\label{penultimo}
	In the statement of Proposition~\ref{invspec}, in place of condition~\eqref{Ab-13}, one could consider
	the following slightly more general assumption: there exists $R\geq 0$ such that
\begin{equation}
	\label{Ab-13-bis}
-2\beta |s|a(x,s)  \leq D_sa(x,s)(b(x)+s^2){\rm sign}(s) \leq 0,
\end{equation}
for a.e.\ $x \in \Omega$ and all $s\in\R$ with $|s|\geq R$, for some 
measurable function $b:\Omega\to\R$ such that $\nu^{-1}\leq b(x)\leq \nu$, for some $\nu>0$.
This condition is satisfied for instance by $a(x,s)=(b(x)+s^2)^{-\beta}$ with $b$ measurable
and bounded between positive constants.
\end{remark}

\begin{remark}\rm
When the maps $s\mapsto j^\sharp(x,s,\xi),j_s^\sharp(x,s,\xi),j_\xi^\sharp(x,s,\xi)$
are bounded, the variational formulation of 
\eqref{probmoddd} can be meant in the sense of distributions (see Proposition~\ref{soldistr}). 
For instance, as it can be easily verified, this occurs for the $a$ mentioned in Remark~\ref{penultimo},
$a(x,s)=(b(x)+s^2)^{-\beta}$. 
\end{remark}

\bigskip
\noindent
{\bf Acknowledgments.} The second author wishes to thank Marco Degiovanni for useful discussions and Luigi Orsina
for some feedback on a preliminary version of the manuscript.

\bigskip


\bigskip

\begin{thebibliography}{99}


\bibitem{albofeortr}
{\sc A.\ Alvino, L.\ Boccardo, V.\ Ferone, L.\ Orsina, G.\ Trombetti}, 
Existence results for nonlinear elliptic equations with degenerate coercivity, 
{\em Ann. Mat. Pura Appl.} {\bf 182} (2003), 53--79.

\bibitem{AR}
{\sc A.~Ambrosetti, P.H.~Rabinowitz},
Dual variational methods in critical point theory and applications,
{\em J.~Funct.~Anal.} {\bf 14} (1973), 349--381.

\bibitem{Ab}
{\sc D.~Arcoya, L.~Boccardo}, Critical points for multiple
integrals of the calculus of variations, {\em Arch. Rational Mech.
Anal.} {\bf 134} (1996), 249--274.

\bibitem{Ab1}
{\sc D.~Arcoya, L.~Boccardo}, Some remarks on critical point
theory for nondifferentiable functionals, {\em NoDEA Nonlinear
Differential Equations Appl.}  {\bf 6} (1999),  79--100.

\bibitem{AbO}
{\sc D.\ Arcoya, L.\ Boccardo, L.\ Orsina}, Existence of critical
points for some noncoercive functionals, {\em Ann.\ Inst.\ H.\
Poincar\'e Anal. Non Lin\'eaire}  {\bf 18}  (2001), 437--457

\bibitem{BoBr}
{\sc L.~Boccardo, H.~Brezis}, 
Some remarks on a class of elliptic equations with degenerate coercivity,
{\em Boll. Unione Mat. Ital. Sez. B Artic. Ric. Mat.} {\bf 6} (2003), 521--530.

\bibitem{BDO}
{\sc L.~Boccardo, A.~Dall'Aglio, L.~Orsina},
Existence and regularity results for some elliptic equations with degenerate coercivity,
{\em Atti Sem.~Mat.~Fis.~Univ.~Modena} {\bf 46} (1998), 51--81.

\bibitem{BO}
{\sc L.~Boccardo,  L.~Orsina},
Existence and regularity of minima for integral functionals noncoercive in the energy space,
{\em Ann.~Scuola Norm.~Sup.~Pisa Cl.~Sci.} {\bf 25} (1997), 95--130.

\bibitem{canino} 
{\sc A. Canino}, 
Multiplicity of solutions for quasilinear elliptic equations, 
{\em Topological Methods Nonlinear Anal.} {\bf 6} (1995), 357--370.

\bibitem{candeg} 
{\sc A. Canino, M. Degiovanni},
Nonsmooth critical point theory and
quasilinear elliptic equations, A. Granas, M. Frigon, G. Sabidussi (Eds.), 
Topological Methods in Differential Equations and Inclusions (Montreal,1994), 
NATO ASI Series, Kluwer Academic Publishers, Dordrecht, (1995),1--50.

\bibitem{CDM}
{\sc J.-N. Corvellec, M.~Degiovanni, M.~Marzocchi},
Deformation properties for continuous functionals and critical point theory,
{\em Topol.~Methods Nonlinear Anal.} {\bf 1} (1993), 151--171.

\bibitem{DMOP}
{\sc G. Dal Maso, F. Murat, L. Orsina, A. Prignet},  Renormalized
solutions of elliptic equations with general measure data 
{\em Ann. Scuola Norm. Sup. Pisa Cl. Sci} {\bf 28} (1999), 741--808

\bibitem{defjfpta}
{\sc M. Degiovanni}, On topological and metric critical point theory,
{\em J. Fixed Point Theory Appl.} {\bf 7} (2010), 85--102

\bibitem{DM}
{\sc M.~Degiovanni, M.~Marzocchi},
A critical point theory for nonsmooth functionals,
{\em Ann. Mat. Pura Appl.} {\bf 167} (1994), 73--100.

\bibitem{DZ}
{\sc M. Degiovanni, S. Zani}, 
Euler equations involving nonlinearities without growth conditions, 
{\em Potential. Anal.} {\bf 5} (1996), 505--512.

\bibitem{pellacci}
{\sc B. Pellacci}, 
Critical points for non-differentiable functionals, 
{\em Boll. UMI B} {\bf 11} (1997), 733--749.

\bibitem{pellacci-nos}
{\sc B.\ Pellacci}, 
Critical points for some functionals of the calculus of variations,
{\em Topol. Methods Nonlinear Anal.} {\bf 17} (2001), 285--305.

\bibitem{pelsqu}
{\sc B. Pellacci, M. Squassina}, Unbounded critical points for a
class of lower semicontinuous functionals, {\em J.\ Differential
Equations} {\bf 201} (2004), 25--62.

\bibitem{toulouse}
{\sc M. Squassina}, 
Weak solutions to general Euler's equations via non-smooth critical point theory,  
{\em Ann. Fac. Sci. Toulouse Math.} {\bf 9} (2000) 113--131.

\bibitem{monog}
{\sc M. Squassina}, 
Existence, multiplicity, perturbation, and 
concentration results for a class of quasi-linear elliptic problems,  
{\em Electron. J. Differential Equations,}  Monograph {\bf 7} 2006, +213 pages, TX, USA.


\end{thebibliography}
\end{document}